\documentclass[a4paper, 12pt]{article}
\usepackage{amsmath, amsfonts, amssymb, amsthm}
\usepackage[english]{babel}

\newcounter{num}[section]

\newenvironment{theorem}
{\refstepcounter{num}%
\bigskip\noindent\nopagebreak[4]{\bf Theorem~\arabic{section}.\arabic{num}. }\it}

\newenvironment{lemma}
{\refstepcounter{num}%
\bigskip\noindent\nopagebreak[4]{\bf Lemma~\arabic{section}.\arabic{num}. }\it}

\newcommand{\LL}{{\mathcal{L}}}

\newcommand{\Ss}{{\mathcal{S}}}
\newcommand{\V}{{\mathrm{V}}}

\renewcommand{\t}{{\tau}}
\newcommand{\s}{{\sigma}}

\newcommand{\M}{{\mathcal{M}}}

\begin{document}

\author{Artem N. Shevlyakov}
\title{On disjunction of equations in the semigroup language with no constants}

\maketitle

\abstract{A semigroup $S$ is an equational domain if any finite union of algebraic sets over $S$ is algebraic. We prove that every nontrivial semigroup in the standard language $\{\cdot\}$ is not an equational domain.}

\section*{Introduction}

It follows from commutative algebra that the union $Y$ of two algebrais sets  $Y_1,Y_2$ over a field $k$ is algebraic again (i.e. $Y$ is a solution of a system of algebraic equations). In the category of groups 
there also exist examples, where any finite union of algebraic sets over a group $G$ is algebraic. Following~\cite{uniTh_IV}, such groups are called {\it equational domains}. 

Notice that any group $G$ (as an algebraic structure) can be considered in three different languages:
\begin{enumerate}
\item $\LL_1=\{\cdot,^{-1},1\}$;
\item $\LL_G=\{\cdot,^{-1},1\}\cup\{g|g\in G\}$;
\item $\LL_H=\{\cdot,^{-1},1\}\cup\{h|h\in H\leq G\}$,
\end{enumerate} 
where the last two languages are the extensions of the first one by constants. An equation over a group is an expression $w(X)=1$, where $w(X)$ is a product of the letters $X\cup X^{-1}$ and language constants. Hence, any group language $\LL\in\{\LL_1,\LL_G,\LL_H\}$ defines its class of algebraic sets $AS(\LL,G)$ over a group $G$. Obviously, the next inclusions holds $AS(\LL_1,G)\subseteq AS(\LL_H,G)\subseteq AS(\LL_G,G)$.

It is possible that a group $G$ is not an equational domain in the languages $\LL_1$ or $\LL_H$, however, $G$ is an equational domain in the rich language $\LL_G$.

Indeed, if we consider groups in the language with no constants $\LL_1$, the search of equational domains among such groups is vain, since the next holds.

\medskip

\noindent{\bf Statement.}\textup{~\cite{uniTh_IV}}
{\it Any nontrivial group of the language $\LL_1$ is not an equational domain.}

\medskip

In the current paper we prove the similar result for semigroups, however its proof is more complicated than in group case (Theorem~\ref{th:main}).

\section{Definitions of semigroup theory}

Let us give the main definitions of semigroup theory, for more details one can recommend~\cite{howie}.

{\it A semigroup} is a non-empty set with an associative binary operation $\cdot$ which is called a \textit{multiplication}. 

Let $a$ be an arbitrary element of a semigroup $S$. {\it A period} $n$ of $a$ is the number of elements in the subsemigroup generated by $a$. For an element $a$ with a period $n$ one can define an {\it index} as the minimal number $m$ such that $a^m=a^n$. 

The elements $a,b\in S$ {\it commutes} if $ab=ba$. A semigroup is {\it trivial} if it contains a single element.

An element $a\in S$ is an \textit{idempotent} if $aa=a$. A semigroup $S$ is a {\it band} (or  \textit{idempotent}) if all its elements are idempotents. A band $S$ is called \textit{rectangular} if there holds the identity $xyz=xz$ for all $x,y,z\in S$. 

A semigroup is \textit{nowhere commutative} if any pair of its distinct elements does not commute. The next theorem describes nowhere commutative semigroups.

\begin{theorem}
\label{th:nowhere_commutative}
A semigroup $S$ is nowhere commutative iff $S$ is a rectangular band.
\end{theorem}

\section{Algebraic geometry over semigroups}

We shall consider semigroups in the language $\LL=\{\cdot\}$. 

Denote by $X$ the finite set of variables $x_1,x_2,\ldots,x_n$. A \textit{term} of the language $\LL$ in variables $X$ is a finite product of variables from the set $X$.

An {\it equation} over the language $\LL$ is an equality of two terms $\t(X)=\s(X)$. For example, the expressions  $x_1x_2^2=x_3x_1$, $x_1^2x_2x_3=x_1x_3$ are equations over $\LL$. A {\it system of equations} (a {\it system} for shortness) is an arbitrary set of equations.
 
The \textit{solution set} of a system $\Ss$ in the semigroup $S$ is naturally defined and denoted by $\V_S(\Ss)$. A set $Y\subseteq S^n$ is {\it algebraic} over a semigroup $S$ if there exists a system $\Ss$ in variables $x_1,x_2,\ldots,x_n$ with the solution set $Y$. 

Following~\cite{uniTh_IV}, let us give the main definition of the paper. A semigroup $S$ is called an  {\it equational domain (e.d.)} if any finite union $Y=Y_1\cup Y_2\cup\ldots\cup Y_n$ of algebraic sets $Y_i$ is algebraic.

\section{Main result}

Our aim is to prove Theorem~\ref{th:main} following the next plan. We consequently prove the absence of equational domains in the next classes of nontrivial semigroups:
\begin{enumerate}
\item idempotent semigroups (Lemma~\ref{l:idempotent});
\item semigroups with the identity $x^2=x^3$ (Lemma~\ref{l:for_x^2=x^3});
\item semigroups with an element $a$ such that $a^2\neq a^3$ (Lemma~\ref{l:first}).
\end{enumerate}

\begin{lemma}
\label{l:idempotent}
Any nontrivial idempotent semigroup $S$ is not an e.d.
\end{lemma}
\begin{proof}
Denote $X=\{x_1,x_2,x_3\}$ and assume the set
\[
\M=\V_S(x_1=x_2)\cup\V_S(x_1=x_3)=\{(x_1,x_2,x_3)|x_1=x_2\mbox{ or } x_1=x_3\}
\]
is algebraic, i.e. it coincides with the solution set of a system $\Ss(X)$. 

There are exactly two cases.

\begin{enumerate}
\item A semigroup $S$ is nowhere commutative. By Theorem~\ref{th:nowhere_commutative} $S$ is an rectangular band. Let $a,b\in S$ be a pair of two distinct elements and $c=ab$. By the identities $xyz=xz$ $xx=x$, we have $ac=c$, $ca=a$.

As $(a,c,c)\notin\M$, there exists an equation $\pi(X)=\rho(X)\notin\Ss(X)$ which does not satisfy this point. Since the set $\{a,c\}$ is a subsemigroup in $S$, the values of the terms $\pi(X)$, $\rho(X)$ at the point $(a,c,c)$ belong to $\{a,c\}$. Without loss of generality one can state $\pi(a,c,c)=a$, $\rho(a,c,c)=c$. Hence, $\pi(X)$ ends by the variable $x_1$, whereas the term $\rho(X)$ ends either by $x_2$ or $x_3$.

Thus, $\pi(a,a,c)=\pi(a,c,a)=a$, and either $\rho(a,c,a)=c$ (if $\rho(X)$ ends by $x_2$) or $\rho(a,a,c)=c$ (if $\rho(X)$ ends by $x_3$). Anyway, we came to the contradiction, since $(a,a,c),(a,c,a)\in\M$.

\item  There exist elements $a,b\in S$ with $ab=ba$ and $c=ab$. It is easy to prove that $ac=ca=c$. As $(a,c,c)\notin\M$, there exists an equation $\pi(X)=\rho(X)\notin\Ss(X)$ which does not satisfy this point. Since the set $\{a,c\}$ is a subsemigroup in $S$, the values of the terms $\pi(X)$, $\rho(X)$ at the point $(a,c,c)$ belong to $\{a,c\}$. Without loss of generality one can state $\pi(a,c,c)=a$, $\rho(a,c,c)=c$. Hence, $\pi(X)$ contains only $x_1$, whereas $\rho(X)$ contains $x_2$ or $x_3$.

Thus, $\pi(a,a,c)=\pi(a,c,a)=a$, and either $\rho(a,c,a)=c$ (if $x_2$ occurs in $\rho(X)$) or $\rho(a,a,c)=c$ (if $x_2$ occurs in $\rho(X)$). Anyway, we came to the contradiction, since $(a,a,c),(a,c,a)\in\M$.

\end{enumerate}
\end{proof}

\begin{lemma}
\label{l:for_x^2=x^3}
Suppose the identity $x^2=x^3$ holds in a semigroup $S$ and $S$ is not idempotent. Hence, $S$ is not an e.d.
\end{lemma}
\begin{proof}
Denote $X=\{x_1,x_2,x_3\}$ and assume the set
\[
\M=\V_S(x_1=x_2)\cup\V_S(x_1=x_3)=\{(x_1,x_2,x_3)|x_1=x_2\mbox{ or } x_1=x_3\}
\]
is algebraic, i.e. it coincides with the solution set of a system $\Ss(X)$. 

As the semigroup $S$ is not idempotent, there exists an element $a\in S$ with $a\neq a^2$. 

As $(a,a^2,a^2)\notin\M$, there exists an equation $\pi(X)=\rho(X)\notin\Ss(X)$ which does not satisfy this point. Since the set $\{a,a^2\}$ is a subsemigroup in $S$, the values of the terms $\pi(X)$, $\rho(X)$ at the point $(a,a^2,a^2)$ belong to $\{a,a^2\}$. Without loss of generality one can state  $\pi(a,a^2,a^2)=a$, $\rho(a,a^2,a^2)=a^2$. Hence, $\pi(X)=x_1$, whereas the term $\rho(X)$ contains either the variables from the set $\{x_2,x_3\}$ or $x_1$ occurs in $\rho(X)$ at least two times.

Thus, $\pi(a,a,a^2)=\pi(a,a^2,a)=a$, and either $\rho(a,a^2,a)=a^2$ or $\rho(a,a,a^2)=a^2$. Anyway, we came to the contradiction, since $(a,a,a^2),(a,a^2,a)\in\M$.
\end{proof}

\begin{lemma}
\label{l:first}
If a semigroup $S$ does not satisfy the identity $x^2=x^3$ it is not an e.d.
\end{lemma}
\begin{proof}
Denote $X=\{x_1,x_2,x_3,x_4\}$. Assume that the solution set of a system $\Ss(X)$ equals 
\[
\M=\V_S(x_1=x_2)\cup\V_S(x_3=x_4)=\{(x_1,x_2,x_3,x_4)|x_1=x_2\mbox{ or }x_3=x_4\},
\]
and $\t(X)=\s(X)$ is an equation of $\Ss$. By the condition of the lemma, there exists an element $a\in S$ with $a^2\neq a^3$ (and, moreover, $a\neq a^2$). 

By the choice of the system $\Ss$, the points
\begin{equation*}
P_1=(a^2,a,a,a),\; P_2=(a,a,a^2,a)
\end{equation*}  
satisfy the equation $\t(X)=\s(X)$.

Suppose the term $\t(X)$ contains $n_i$ occurrences of the variable $x_i$. Similarly, $m_i$ is the number of occurrences of the variable $x_i$ in $\s(x)$. The equalities $\t(P_i)=\s(P_i)$ imply 
\begin{equation*}
\begin{cases}
a^{2n_1+n_2+n_3+n_4}=a^{2m_1+m_2+m_3+m_4},\\
a^{n_1+n_2+2n_3+n_4}=a^{m_1+m_2+2m_3+m_4},\\
\end{cases}
\end{equation*}

Let us multiply the equations of the system above and obtain the equality
\begin{equation}
a^{3n_1+2n_2+3n_3+2n_4}=a^{3m_1+2m_2+3m_3+2m_4}.
\label{eq:power_from13}
\end{equation}

Consider a point $Q=(a^3,a^2,a^3,a^2)\notin \M$. Compute
\begin{eqnarray*}
\t(Q)=a^{3n_1+2n_2+3n_3+2n_4},\\
\s(Q)=a^{3m_1+2m_2+3m_3+2m_4}.
\end{eqnarray*} 

By~(\ref{eq:power_from13}), we have $\t(Q)=\s(Q)$. As $\t(X)=\s(X)$ was chosen as an arbitrary equation of $\Ss$, $Q\in\V_S(\Ss)$ that contradicts with the choice of the system $\Ss$.
\end{proof}

The main result immediately follows from Lemmas~\ref{l:idempotent},~\ref{l:for_x^2=x^3},~\ref{l:first}.

\begin{theorem}
\label{th:main}
Any nontrivial semigroup in the language $\LL=\{\cdot\}$ is not an equational domain.
\end{theorem}

The information of the author:

Artem N. Shevlyakov

Omsk Branch of Institute of Mathematics, Siberian Branch of the Russian Academy of Sciences

644099 Russia, Omsk, Pevtsova st. 13

Phone: +7-3812-23-25-51.

e-mail: \texttt{a\_shevl@mail.ru}


\begin{thebibliography}{20}

\bibitem{uniTh_IV}
E.~Daniyarova, A.~Miasnikov, V.~Remeslennikov, Algebraic geometry over algebraic structures IV: equational domains and co-domains, Algebra \& Logic, 49, 6, 2010, 715--756.

\bibitem{howie}
J.M.~Howie, Fundamentals of Semigroup Theory. Oxford: Clarendon Press, 1995, 351~p.


\end{thebibliography}
\end{document}